\documentclass[a4paper,12pt]{article}
\usepackage{amssymb, amsmath, afterpage}
\usepackage{dynkin-diagrams}

\usepackage{array,multirow,enumitem} 
\usepackage{epic}
\usepackage{amsfonts,amsthm}
\newcommand{\pli}[1]{\ensuremath{\overset{+}{\imath}}}

\newcommand{\Sym}{\ensuremath{\mathrm{Sym}}}
\usepackage{xcolor}
\newcommand{\diam}{\ensuremath{\mathrm{Diam}}}
\newcommand{\rank}{\ensuremath{\mathrm{rank}}}

\newcommand{\A}[1]{\ensuremath{\mathrm{A}}_{#1}}
\newcommand{\B}[1]{\ensuremath{\mathrm{B}}_{#1}}
\newcommand{\D}[1]{\ensuremath{\mathrm{D}}_{#1}}
\newcommand{\F}[1]{\ensuremath{\mathrm{F}}_{#1}}
\renewcommand{\H}[1]{\ensuremath{\mathrm{H}}_{#1}}
\newcommand{\I}{\ensuremath{\mathcal{I}}}
\newcommand{\ZZ}{\mathrm{Z}}

\newcommand{\R}{R} 
\newcommand{\E}{\ensuremath{\mathcal{E}_0}}
\newcommand{\Eh}{\ensuremath{\widehat{\mathcal{E}}_0}}

\renewcommand{\P}{\mathcal{P}}

\newtheorem{thm}{Theorem}[section]
\newtheorem{lemma}[thm]{Lemma}

\newtheorem{cor}[thm]{Corollary}
\theoremstyle{definition}
\newtheorem{defn}[thm]{Definition}
     
\begin{document}
    \title{\Large\textbf{The Excess Zero Graph of a Coxeter Group}}
      \author{Sarah Hart, Veronica Kelsey and Peter Rowley} \date{}
  \maketitle

\vspace*{-10mm}
\begin{abstract} For a Coxeter group $W$ with length function $\ell$, the {\em excess zero graph} $\E(W)$ has vertex set the non-identity involutions of $W$, with two involutions $x$ and $y$ adjacent whenever $\ell(xy)=\ell(x)+\ell(y)$. Properties of this graph such as connectivity, diameter and valencies of certain vertices of $\E(W)$ are explored.\\


\end{abstract}
\section{Introduction}\label{sec:intro}
The study of Coxeter groups goes hand in hand with the study of their length functions. Indeed, a group generated by involutions with an associated length function is a Coxeter group precisely when it satisfies the Exchange Condition (see Lemma \ref{lem:ex}). Many results about Coxeter groups involve the interplay between the group multiplication and the length function. In this paper, we study an aspect of this interplay in the case of involutions.

Let $W$ be a Coxeter group with $\R$ its set of fundamental (or simple) reflections, and let $w \in W$. The \emph{length function} $\ell$ on $W$ is defined by $\ell(w)=0$ if $w=1$, and for $w \neq 1$ 
\[\ell(w)=\min\{k \in \mathbb{N} \; |\; w=r_1r_2 \cdots r_k \text{ where } r_i \in \R\}.\]
The nature of the relationship between $\ell(w_1w_2)$ and $\ell(w_1)+\ell(w_2)$, where $w_1,w_2 \in W$, arises in many situations. Generally, it is not the case that $\ell(w_1w_2)=\ell(w_1)+\ell(w_2)$, though there are a number of instances where this does occur. For example, if $W_J$ is a standard parabolic subgroup, then there is a set $X_J$ referred to as the set of distinguished right coset representatives of $W_J$ in $W$ for which $\ell(wx)=\ell(w)+\ell(x)$ for all $w \in W_J$, $x \in X_J$ (see Lemma~\ref{lem:cosetreps}).

The involutions in a group are often important in highlighting certain properties of the group. This is especially true for Coxeter groups, particularly because of the fundamental reflections. For a Coxeter group $W$ put 
\[\mathcal{W}=\{w \in W\;|\; x,y \in W, \; w=xy, \; x^2=1=y^2\}.\]
By \cite{carter}, if $W$ is a finite Coxeter group, then $W=\mathcal{W}$, but in general $\mathcal{W}$ may be a proper subset of $W$. For $w \in \mathcal{W}$, the \emph{excess} of $w$, denoted $e(w)$, was introduced in \cite{invol} and is defined by 
\[ e(w)=\min\{\ell(x)+\ell(y)-\ell(w) \; | \; x, y \in W, \; w=xy, \; x^2=y^2=1\}.\]
Evidently, $e(w)=0$ if and only if there exist $x, y \in W$ with $w=xy$ such that $x^2=y^2=1$ and $\ell(w)=\ell(x)+\ell(y)$. Note that, as $x$ or $y$ could be the identity, $e(w)=0$ for all involutions in $W$. 

In a series of papers \cite{invol, excessminlen, onexcess, excessmaxlen}, properties of excess are explored. In \cite{invol} it is shown that for any finite rank Coxeter group $W$ and $w \in \mathcal{W}$, there exists $w_{\ast} \in W$ which is $W$-conjugate to $w$ and has $e(w_{\ast})=0$. In a similar vein, in \cite{excessminlen}, it is proved that for finite Coxeter groups such a $w_{\ast}$ may be found with the property that it has minimal length in its ${W}$-conjugacy class. An analogous result for maximal length in $W$-conjugacy classes is obtained in \cite{excessmaxlen}.
These investigations into elements of $W$ with zero excess have to some extent pushed involutions into the background. In this paper we put the spotlight on involutions and their impact on zero excess.

\begin{defn} Let $W$ be Coxeter group with length function $\ell$ and set of non-identity involutions $\I(W)$. The {\em excess zero graph} $\E(W)$ has vertex set $\I(W)$ with $x,y \in \I(W)$ adjacent whenever $\ell(xy)=\ell(x)+\ell(y)$.
\end{defn}

The definition of this graph first arose in the study of prefixes in Coxeter groups, see \cite{prefixes}. For $w \in W$, we say that $u \in \I(W)$ is an {\em involution prefix} of $w$ if there exists $v \in W$ such that $w=uv$ with $\ell(w)=\ell(u)+\ell(v)$. An element $w$ has the {\em ancestor property} if the set of all involution prefixes for $w$ contains a unique involution of maximal length. It is conjectured in \cite{prefixes} that the ancestor property holds for all non-identity elements in finite Coxeter group. A result of this conjecture would be a canonical way to write each non-identity element $w$ of $W$ as $w=x_1x_2\cdots x_k$ with $x_i \in \I(W)$ and $\ell(w)=\ell(x_1)+\cdots + \ell(x_k)$. 
In the language of the excess zero graph, this means that $x_1, x_2,\ldots,x_k$ is a path in $\E(W)$. It was this observation that first motivated the study of the excess zero graph as a potential tool for proving the ancestor conjecture.

We briefly justify our choice of vertex set.
In its current form, the excess zero graph has no loops. If the identity element were included in the vertex set, then it would be adjacent to all other vertices, including itself.
In addition to the observation above on ancestors, by considering only involutions as vertices the excess zero graph is undirected.\\ \medskip

As an indicative example, in Table \ref{table:tabA}, we take a glance at $\E(W)$ when $W = W(\A{n})$ and $n \in \{3,4,5,6\}$. The entries of the table were calculated with the aid of \textsc{Magma} \cite{magma}. By $i^k$ we mean that there are $k$ involutions in $\E(W)$ which have valency $i$. 
\begin{table}[h]
\begin{tabular}{|c|c|}
\hline
Group & Valency distribution\\
\hline
$W(\A{3})$ & $0^1.1^3.2^1.3^1.4^3$\\
$W(\A{4})$ & $0^1.1^4.2^2.3^6.4^3.6^3.7^2.12^4$\\
$W(\A{5})$ & $0^1.1^5.2^3.3^{13}.4^7.5^2.7^{15}.8^1.9^8.10^1.13^4.15^1.19^9.37^5$\\
$W(\A{6})$ & $0^1.1^6.2^4.3^{23}.4^9.5^6.6^9.7^{22}.8^4.9^{27}.10^4.11^{10}.12^7.13^8.14^2.15^{10}.16^3.19^{12}.21^{4}.22^{12}.24^2.25^5$\\
& $27^8.29^2.31^4.39^2.41^5.55^4.59^{10}115^6$\\
\hline
\end{tabular} \caption{Valency distribution for small groups of type $\A{n}$} \label{table:tabA}
\end{table}

At first sight the data in Table \ref{table:tabA} looks to be very haphazard, but this is not totally unexpected. Of course, $W$ acts upon $\I(W)$ by conjugation, but conjugation severely damages the excess zero property. As a result $W$ does not induce graph automorphisms on $\E(W)$. We do however observe some structure. For example, for $n=3, 4, 5, 6$, compare the final term of the valency distribution in $W(A_n)$ (respectively, 4, 12, 37, and 115) with $|\I(W(\A{n}))|$, which is, respectively, 9, 25, 75 and 231. This is part of a general pattern between the highest valency and the number of vertices of this graph which we see in Corollary~\ref{thm:highval}.
The one isolated vertex of $\E(W)$ in Table \ref{table:tabA} (indicated by $0^1$) is $w_0$, the longest element of $W$(and only defined when $W$ is finite). \\

The main topics to be explored in this paper are the possible diameters of $\E(W)$ and the distribution of the valencies of its vertices. We outline our results below, beginning with a theorem concerned with the first of these topics.
\begin{thm}\label{thm:diam} Let $W$ be a Coxeter group of rank at least 2.
\begin{enumerate}[label=\rm{(\roman*)}]
\item If $W$ is finite, then $\E(W)$ has two connected components: $\I(W)\backslash \{w_0\}$ and $\{w_0\}$. If $W$ is infinite, then $\E(W)$ is connected.
\item Let $\Eh(W)$ be the connected component of $\E(W)$ not containing $w_0$. Then the diameter of $\Eh(W)$ is at most 3.
\end{enumerate}
\end{thm}
The requirement of rank being at least 2 is to omit $W=\A{1}$, since in this case $\I(W)=\{w_0\}$ and so $\Eh(W)$ is the empty graph, whose diameter is undefined.

We now consider the valencies of vertices. For $x \in \I(W)$, let $\Delta_1(x)$ be the set of vertices adjacent to $x$ in $\E(W)$. Hence, the valency of $x$ is the cardinality of $\Delta_1(x)$. We have the following result for finite rank Coxeter groups.

\begin{thm}\label{thm:samecard}  Suppose that $W$ is a Coxeter group of finite rank and let $r$ be a fundamental reflection of $W$. Then the sets $\Delta_1(r)$ and $\mathcal{I}(W) \setminus (\{r\} \cup \Delta_1(r) )$ have the same cardinality.
\end{thm}

A corollary of Theorem~\ref{thm:samecard}, hinted at in Table \ref{table:tabA}, is as follows.

\begin{cor}\label{thm:highval} For a finite Coxeter group $W$ and $r$ a fundamental reflection,
\[|\Delta_1(r)|=\frac{|\I(W)|-1}{2}.\]
Moreover, for $x \in \mathcal{I}(W) \setminus R, $ \[|\Delta_1(x)|  < \frac{|\I(W)|-1}{2}.\]
\end{cor}

Corollary~\ref{thm:highval} is only the tip of the iceberg for finite irreducible Coxeter groups $W$, as we see with our next theorem. For positive integers $m$ and $n$, and $x \in \I(W(\A{n-1})) = \I(\Sym(n))$ a product of $m$ distinct mutually commuting fundamental reflections, write $\delta(m,n)$ for $|\Delta_1(x)|$. We shall prove that this is well-defined in Lemma~\ref{wlog}. Even though $1\notin \I(W)$, it is useful to define $\delta(0,n) = |\I(\Sym(n))|$; we also set $\delta(m,0) = 0$. 
\begin{thm} \label{thm:valency}
Let $m\geq 2$ and $n\geq 2m$. Then  
\[\delta(m,n) = \textstyle\frac{1}{2}\big(\delta(m-1,n) + (m-1)\delta(m-2,n-4) + m-2\big).\]
\end{thm}

At the opposite end of the spectrum to Corollary~\ref{thm:highval}  we have those $w \in W$ for which $|\Delta_1(w)| = 1$. We call such $w$ \emph{pendant elements}. Let $\P(W)$ denote the set of pendant elements of $W$. Our next theorem classifies the pendant elements of finite irreducible Coxeter groups. For the notation employed in this theorem and the labelling of the relevant Dynkin diagrams, see Section~\ref{sec:background}.

\begin{thm}\label{thm:pendantelts} Let $W$ be a finite irreducible Coxeter group of rank $n \geq 2$ with \linebreak
	$R =\{ r_1, r_2, \ldots, r_n \}$.
\begin{enumerate}[label=\rm{(\roman*)}]
		\item Suppose $w_0\in \ZZ(W)$. (This occurs if and only if $W$ is of type $\B{n}$, type $\D{n}$ with $n$ even, type $\F{4}$,  $\H{3}$, $\H{4}$, $\mathrm{E}_7$ or $\mathrm{E}_8$, or type $\mathrm{I}_2(m)$ with $m$ even.) Then $\P(W) = \{ w_0r \; | \; r \in R \}$.  
		\item If $W$ is of type $\A{n}$ and $n \geq 2$,  setting $J =  \{1,2,\ldots, \lceil \frac{n}{2} \rceil\}$ we have 
		                 \[ \P(W) = \big\{w_0[ i \nearrow (n+1-i)] ,w_0 [(n+1-i) \searrow i]  \; | \; i \in J\big\}.\]
		\item If $W$ is of type $\D{n}$ where $n $ is odd, then
		               \[ \P(W) = \big\{ w_0r_i,  w_0r_nr_{n-2}r_{n-1},   w_0r_{n-1}r_{n-2}r_n \; | \; i \in \{1,  \dots , n - 2\} \big\}.\]
	    \item If $W$ is of type $\mathrm{E}_6$, then  
	    \[\P(W) = \{w_0r_2, w_0r_4, w_0r_5r_4r_3, w_0r_3r_4r_5, w_0r_6r_5r_4r_3r_1,  w_0r_1r_3r_4r_5r_6 \}.\]
	    	\item If $W$ is of type $ \mathrm{I}_2(m)$,  with $m$ odd and $m \geq 5$,  then $\P(W) = \{w_0r_1r_2, w_0r_2r_1 \}$.  	          
\end{enumerate}
\end{thm}

\begin{cor} \label{cor:lwn} If $W$ is a finite irreducible Coxeter group of rank $n$, then $|\P(W)| = n$.
\end{cor}

Observe that the pendant elements in Theorem \ref{thm:pendantelts} have the form $w_0x$ where $x^{w_0}=x^{-1}$. These elements are examples of twisted involutions. In general, for $w \mapsto w^{\ast}$ a self-inverse automorphism on $W$ preserving $R$, we can form a twisted Coxeter system $(W,R,\ast)$ with a set of \emph{twisted involutions} $\I_{\ast}=\{w \in W \; | \; w^{\ast}=w^{-1}\}$. These twisted systems have many interesting properties, see for example \cite{Hansson}, \cite{Hultman} and \cite{Marberg}. \medskip

We end this section with an overview of this paper. Section~\ref{sec:background} sets up notation for the paper, and also reviews some well known results that we will employ. Not surprisingly, these results focus on properties of the length function.

Section~\ref{sec:diam} looks at the connectedness of $\E(W)$, first proving Theorem~\ref{thm:diam}, and then determining the various possible diameters of $\Eh(W)$.  For example, in Theorem~\ref{thm:finitediameter} we see that the diameter is $3$ in the case when either $W$ is a compact hyperbolic Coxeter group or an affine Coxeter group (other than $W(\widetilde{A}_1)$), while Lemmas~\ref{lem:infinitered} and \ref{infiniteredred} examine some classes of infinite Coxeter groups which have diameter $2$. 

Valencies of particular elements of a Coxeter group are the subject of Section~\ref{sec:val}. The proof of Theorem~\ref{thm:samecard} is relatively quick, in contrast to the proof of Theorem~\ref{thm:valency}, which determines $|\Delta_1(x)|$ where $x$ is a product of commuting fundamental reflections, and $W$ is a Coxeter group of type $\A{n}$. First, in Lemma~\ref{wlog}, we show  that (perhaps surprisingly) two conjugate involutions of $\Sym(n)$ which have minimal length in their conjugacy class also have the same valency. Then the recursive formula displayed in Theorem~\ref{thm:valency} is established. 

Pendant elements in finite irreducible Coxeter groups $W$ are classified in Section~\ref{sec:pendant}, when the proofs of Theorem~\ref{thm:pendantelts} and Corollary~\ref{cor:lwn} are given. When $w_0 \in \ZZ(W)$ the pendant elements are quickly located (see Lemma~\ref{lem:pendantw0central}). The case of $w_0 \notin \ZZ(W)$ requires a more involved analysis. Lemma~\ref{lem:pendantAn} settles the case when $W$ is of type $\A{n}$, while Subsection~\ref{subsec:pendantDn} is devoted to dealing with $W$ of type $\D{n}$ with $n$ odd. 

Our last section gives the valency distribution for four of the small exceptional finite Coxeter groups.

\section*{Acknowledgements}  This work was begun during a visit to the Mathematisches Forschungsinstitut Oberwolfach, as part of their Oberwolfach Research Fellows program, and completed with funding from the Manchester Institute of Mathematical Sciences.
The second author was funded by the Heilbronn Institute for Mathematical Research.

\section{Background material}\label{sec:background}

Here we take a quick trip through some basic material and notation relating to Coxeter groups which will be needed later. A good general reference for these and other results is \cite{Humph}. Recall that $W$ being a Coxeter group of finite rank $n$ means that it has a finite set $R = \{r_1, \dots, r_n \}$ of fundamental (or simple) reflections such that 
\[W = \langle R \; | \; (r_ir_j)^{m_{ij}} = 1 \;  \text{for all } r_i, r_j \in R \rangle \]
is a presentation for $W$ with $m_{ii} = 1$ and  $m_{ij} \geq 2$ for $i \not= j$. As noted in Section~\ref{sec:intro} we have the length function $\ell$. There is an alternative description of $\ell(w)$ which we now review. Letting $V$ be a vector space over $\mathbb{\R}$ with basis $\Pi=\{\alpha_{r} \; | \; r \in \R\}$ we define  the following bilinear form on $V$:
\[\langle \alpha_{r}, \alpha_{s} \rangle = \begin{cases} -\textrm{cos}(\pi/m_{rs}) & \text{ if } m_{rs}<\infty,\\
-1 & \text{ if } m_{rs}=\infty.\end{cases}\]
We can now define a faithful action of $W$ on $V$ which preserves the bilinear form. For $r \in \R$ and $v \in V$ let 
\[r \cdot v = v-2 \langle v, \alpha_r \rangle \alpha_r.\]
The {\em root system} of $W$ is
$\Phi=\{w \cdot \alpha_r \; | \; w \in W , r \in \R\}$.
The {\em positive roots} are $\Phi^+=\{\sum_{r \in \R} \lambda_r \alpha_r  \in \Phi \; | \; \lambda_r \geq 0 \text{ for all } r \in \R\}$, and the {\em negative roots} are $\Phi^-=-\Phi^+$. Thus $\Phi = \Phi^+ \dot{\cup} \Phi^-$. We define some important subsets of $\Phi^+$. For $w \in W$ let 
\[N(w) = \{\alpha \in \Phi^+ \; | \; w \cdot \alpha \in \Phi^-\},\]
and so $|N(w)|$ is the number of positive roots taken negative by $w$.  A key result on the length function in a Coxeter group is our first lemma.

\begin{lemma}[{{\cite[\S 5.2]{Humph}}}]\label{lem:pos/neg} For $w \in W$ and $r \in R$, either $\ell(wr) = \ell(w) -1$ or
$\ell(wr) = \ell(w) +1$.
\end{lemma}

The close relationship between $\ell(w)$ and $N(w)$ is apparent in the next result.

\begin{lemma} Suppose that $W$ is a Coxeter group with $x,y \in W$. Then the following hold.
\begin{enumerate}[label=\rm{(\roman*)}]
\item $\ell(x)=|N(x)|$
\item $N(xy) =\big[ N(y) \setminus \left(-y^{-1}N(x)\right)\big] \dot\cup y^{-1}\left(N(x)\setminus N(y^{-1})\right)$
\item $\ell(xy) = \ell(x) + \ell(y) - 2|N(x)\cap N(y^{-1})|$
\end{enumerate}
\end{lemma}

\begin{proof} For (i) and (ii) see, respectively,  \cite[§5.6 Proposition(b)]{Humph} and \cite[Lemma 2.2]{invol}. Part (iii) follows from (ii).
\end{proof}

Next we have the so-called Exchange Condition. Recall that an expression for $w$ of the form $w=r_1r_2\cdots r_k$  with $r_i \in R$ is called \emph{reduced} if $k=\ell(w)$.

\begin{lemma}[{{\cite[\S 5.8]{Humph}}}]\label{lem:ex}Let $w = r_1\cdots r_k$ (not necessarily reduced), where each $r_i$ is a simple reflection. If $\ell(wr) < \ell(w)$ for some simple reflection $r$, then there exists an index $i$ for which $wr = r_1 \cdots \hat{r_i}\cdots r_k$ (and thus $w = r_1 \cdots \hat{r_i} \cdots r_k r$, with a factor $r$ exchanged for a factor $r_i$).
In particular, $w$ has a reduced expression ending in $r$ if and only if $\ell(wr) < \ell(w)$.
\end{lemma}

As a consequence of the above, when $W$ is finite, $w_0$ has a reduced expression ending in $r$ for all $r \in \R$.
 
For $J \subseteq \R$ we define the corresponding {\em standard parabolic} subgroup to be $W_J = \langle J \rangle$. Then $W_J$ is a Coxeter group with root system 
\[\Phi_J = \{ w \cdot \alpha_r \; | \; r \in J, \; w \in W_J\}.\]

Suppose $J\subseteq K \subseteq R$. Adapting the notation from \cite[§2.1]{GeckPfeiffer} we define the following sets
\begin{align*} X^K_J &= \{w \in W_K \; | \; \ell(sw) > \ell(w)\;  \mbox{for all} \; s \in J \}\\
^K_JX &= \{w \in W_K \; | \; \ell(ws) > \ell(w)\;  \mbox{for all} \; s \in J \}
\end{align*}
We write $X_J$ for $X^R_J$ and $_JX$ for $^R_JX$. Our next result was alluded to in Section~\ref{sec:intro}.

\begin{lemma}[{{\cite[ Propostion 2.1.1] {GeckPfeiffer}}}] \label{lem:cosetreps} Suppose that $W$ is a Coxeter group and $J \subseteq \R$. Then the following hold.
\begin{enumerate}[label=\rm{(\roman*)}]
\item For each $w \in W$ there is a unique $y \in W_J$ and $x \in X_J$ such that  $w = yx$.  Furthermore, $\ell(yx) = \ell(x) + \ell(y)$. In particular, $X_J$ is a complete set of right coset representatives of $W_J$ in $W$.
\item The set $\{x^{-1} \; | \; x \in X_J \}$ equals $_JX$, and $_JX$ is a complete set of left coset representatives of $W_J$ in $W$.
\end{enumerate}
\end{lemma}
 We call $X_J$ the set of \emph{distinguished right coset representatives}, and $_JX$ the set of \emph{distinguished left coset representatives}, of $W_J$ in $W$. \smallskip

When investigating finite Coxeter groups of types $\A{n}, \D{n}$ and $\mathrm{E}_6$ in Section~\ref{sec:pendant}, we assume their Dynkin diagrams are labelled as in Figure \ref{DynkinFig}.
\begin{figure}[h!bt]
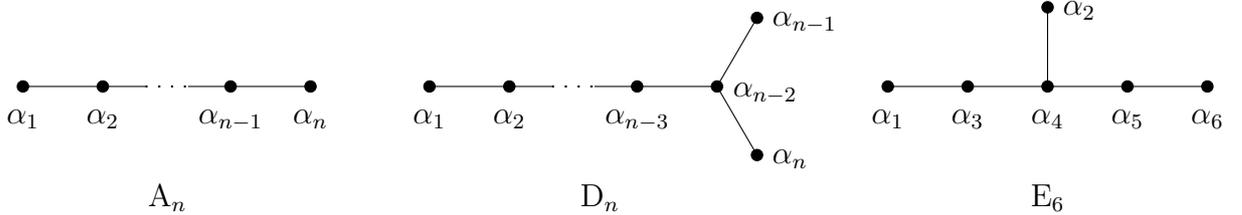

	\begin{tabular}{ccccc}
		$\begin{dynkinDiagram}
			[text style/.style={scale=1}, 
			scale=1.5, 
			edge length=0.7cm, 
			indefinite edge/.style={thick,loosely dotted},
			labels={\alpha_{1},\alpha_{2},\alpha_{n-1},\alpha_n} ]A{}  \end{dynkinDiagram} $
		
		&&
		$\begin{dynkinDiagram}
			[text style/.style={scale=1}, 
			scale=1.5, 
			edge length=0.7cm, 
			indefinite edge/.style={thick,loosely dotted},
			labels={\alpha_1,\alpha_2,\alpha_{n-3},,\alpha_{n-1},\alpha_{n}} ]D{} \end{dynkinDiagram} \hspace{-1.5cm} \alpha_{n-2}$
		
		&&
		$\begin{dynkinDiagram}
			[text style/.style={scale=1}, 
			scale=1.5, 
			edge length=0.7cm,  labels={\alpha_1,\alpha_2,\alpha_3,\alpha_4,\alpha_5,\alpha_6}, ]E{6}\end{dynkinDiagram}$
		\\
		$\A{n}$ && $\D{n}$ && $\mathrm{E}_6$
	\end{tabular}\caption{Labelling of Dynkin diagrams} \label{DynkinFig}
\end{figure}
\medskip

Sometimes we shall write $[i_1,i_2, \dots , i_{j-1},i_j]$ to stand for the product \[r_{i_1}r_{i_2} \cdots r_{i_{j-1}}r_{i_j}\] where  $r_{i_1,}r_{i_2}, \dots, r_{i_{j-1}},r_{i_j} \in R$. In some circumstances we compress this notation yet further: when $i \leq j$ we let 
$[ i \nearrow j]=r_ir_{i+1} \cdots r_{j-1}r_j$ and $ [j  \searrow i]=r_jr_{j-1} \cdots r_{i+1}r_i$. \smallskip

Our next theorem is used when pinning down the pendant elements in Coxeter groups of type $\D{n}$. We use the labelling convention as in Figure \ref{DynkinFig}.

\begin{thm}\label{Dncosetreps} Suppose that $W$ is a Coxeter group of type $\D{n}$. Let $J\coloneqq R \setminus \{ [n] \}$ and $x \in X_J$ a non-identity element. Then $x$ has a reduced expression $x = ab$ where either
\begin{enumerate}[label=\rm{(\roman*)}]
\item $a \in \{ [n],[n,n-2,n-1]\}$ and $b \in W_{\{1,\ldots, n-2\}}$; or
\item $a = [n,n-2,n-3,n-1,n-2,n]$.
\end{enumerate}
\end{thm}

\begin{proof} We argue by induction on $n$, starting with $n = 4$. Employing Algorithm B described in \cite[\S 2.1]{GeckPfeiffer} gives 
\[X_J = \{ 1, [4],[4,2],[4,2,1],[4,2,3], [4,2,1,3],[4,2,1,3,2],[4,2,1,3,2,4] \}.\]
Bearing in mind that $[1]$ and $[3]$ commute we see that either alternative (i) holds or $x = [4,2,1,3,2,4]$, which gives case (ii).  Now assume that $n \geq 5$.  Set $K = R \setminus \{ [1] \}$ and $L = R \setminus \{[1], [n-1] \}$. Then $W_K$ is of type $\D{n-1}$, with $W_L$ a standard parabolic subgroup of $W_K$ of type $\A{n-2}$. Let $x \in X_J$.\smallskip

 Appealing to the Mackey decomposition \cite[Lemma 2.1.9]{GeckPfeiffer},
\[ X_J = \bigcup_{d \in X_J \cap (X_{K}) ^{-1}} dX^{K}_{J^d \cap K}.\]
Again using \cite[Algorithm B, \S 2.1]{GeckPfeiffer} 
\[ X_K = \left\{ 
1,[1], [1 \nearrow 2], \dots, [1 \nearrow n], [1 \nearrow n-2,n], 
[1 \nearrow n,n-2], \dots ,[1 \nearrow n,n-2 \searrow 1] 
\right\}\]
and hence $X_J \cap (X_{K}) ^{-1} = \{d_1 = 1, d_2 = [n,n-2 \searrow 1] \}$. Then we see that $J^{d_2} \cap K = L$ and hence 
\[ X_J = X^K_{J\cap K} \cup d_2 X^K_L.\]
If $x \in X^K_{J \cap K}$, then, as $W_K$ is of type $\D{n-1}$ (with $[n-1], [n]$ in $X^K_{J \cap K}$),  induction yields that either (i) or (ii) hold. Now suppose that $x \in d_2X^K_L.$ Therefore $x = d_2e$ for some $e \in X^K_L$.  Again applying induction, this time to $X^K_L$, we deduce that we have a reduced expression $e=ab$ where either $a \in \{[n-1], [n-1,n-2,n] \}$ with $b$ having no appearances of $[n-1]$ nor $[n]$, or $a = [n-1,n-2,n,n-3,n-2,n-1]$. Assume the first possibility holds, so $a \in \{[n-1], [n-1,n-2,n] \}$. If $a = [n-1]$, then
\begin{align*}
x&=[n,n-2 \searrow 1,n-1]b\\
&=[n,n-2,n-1][n-3 \searrow 1]b,
\end{align*}
which is of the desired form. Meanwhile if $a=[n-1,n-2,n]$, then
\begin{align*}
x &=[n,n-2\searrow 1,n-1,n-2,n]b\\
&=[n,n-2,n-3,n-1,n-2,n][n-4\searrow 1]b
\end{align*}
which is alternative (ii).
In the remaining case we have
\begin{align*}
x 
&=[n,n-2,n-3,n-4 \searrow 1,n-1,n-2,n,n-3,n-2,n-1]b\\
&=[n,n-2,n-3,n-1,n-2,n][n-4 \searrow 1, n-3,n-2,n-1]b
\end{align*}
again yielding alternative (ii), and this completes the proof of the theorem.
\end{proof}

\section{Diameter of $\Eh(W)$}\label{sec:diam}
Recall that $W$ is a Coxeter group with set of fundamental reflections $R$, and non-identity involution set $\I(W)$. The subgraph $\Eh(W)$ of the excess graph $\E(W)$ is defined by taking the connected component which does not contain $w_0$.
We use $\diam(\Eh(W))$ to denote the diameter of $\Eh(W)$.

\begin{lemma}\label{lem:fundneighbour} Let $x \in \I(W)$. Either $x$ is adjacent in $\E(W)$ to some $r \in R$, or $W$ is finite and $x=w_0$.
\end{lemma}
\begin{proof}
	Assume, for a contradiction, that $x$ is not adjacent to any $r$ in $\R$.
	Then $\ell(xr) = \ell(x) - 1$, and so $\alpha_r \in N(x)$ for all $r \in \R$. Since 
	$\Phi^+ = \{\sum_{r \in \R} \lambda_r\alpha_r \in \Phi \; |\; \lambda_r \geq 0 { \text{ for all $r$ in $R$}}\}$, it follows that $\Phi^+\subseteq N(x)$. If $W$ is infinite, then $\Phi^+$ is infinite, whereas $N(x)$ is finite, a contradiction. Therefore $W$ is finite, and now $\Phi^+\subseteq N(x)$ forces $\Phi^+ = N(x)$, which implies $x = w_0$.
\end{proof}

\begin{proof}[Proof of Theorem~\ref{thm:diam}]
If $W$ has a longest element $w_0$ (that is, if $W$ is finite), observe that $\ell(xw_0) = \ell(x) + \ell(w_0)$ only when $x = 1$, and so $\{w_0\}$ is a connected component of $\E(W)$. Let $x$ be any element of $\I(W)$ other than (where it exists) $w_0$. By Lemma~\ref{lem:fundneighbour}, $x$ is adjacent to some $r \in R$. If $r$ and $s$ are distinct elements of $\R$, then $\ell(rs) = 2 = \ell(r) + \ell(s)$ and so $r$ and $s$ are adjacent. Hence
$\Eh(W)$ is connected with diameter at most 3.
\end{proof}

The next results allow us to determine the diameter of $\Eh(W)$ for several classes of Coxeter groups, including finite and affine Coxeter groups.

\begin{lemma}\label{lem:maxparafinite} Let $W$ be a Coxeter group of finite rank at least 3. If there exist distinct fundamental reflections $r$ and $s$ such that $W_{\R \backslash \{r\}}$ and $W_{\R \backslash \{s\}}$ are both finite, then the diameter of $\Eh(W)$ is 3. 
\end{lemma}

\begin{proof}
	Let $X=W_{\R \backslash \{r\}}$ and $Y=W_{\R \backslash \{s\}}$ with longest elements $x$ and $y$ respectively. Then $x, y \in \I(W)$. 
	For all $t \in R \backslash \{r\}$, it follows from the maximal length of $x$ that $\ell(xt)<\ell(x)$. Hence, if $z$ is adjacent to $x$ in $\Eh(W)$, then every reduced expression for $z$ starts with $r$. 
	Similarly, if $z$ is adjacent to $y$ in $\Eh(W)$, every reduced expression for $z$ starts with $s$. It follows that $x$ and $y$ have no common neighbours. Finally, let $t\in R\setminus \{r,s\}$.  
	By the Exchange Condition, there exists a reduced expression for $y$ starting in $t$, and so $x$ and $y$ are not adjacent. Hence, the diameter of $\Eh(W)$ is 3.  
\end{proof}

\begin{lemma}\label{infinitediameter}
	Let $W$ be an infinite Coxeter group. Then $\diam(\Eh(W)) \in \{2,3\}$. If there exists $r \in \R$ such that $m_{rs} = \infty$ for all $s \in \R \backslash \{r\}$, or if $W$ has infinite rank, then $\diam(\Eh(W))=2$.
\end{lemma}
\begin{proof}
	Since $W$ is infinite, $W$ has rank at least 2, and so there exist distinct $r, s$ in $R$. By Theorem~\ref{thm:diam},	$\Eh(W)$ is connected with diameter at most 3. If $r$ commutes with $s$, then $rs$ is an involution in $\Eh(W)$ which is not adjacent to $r$. Otherwise, $rsr$ is an involution in $\Eh(W)$ which is not adjacent to $r$. Therefore $\diam(\Eh(W))\geq 2$. Hence $\diam(\Eh(W)) \in \{2,3\}$. 
	
	Now, suppose there exists $r \in \R$ such that $m_{rs} = \infty$ for all $s \in \R \backslash \{r\}$. For any element $w$ of $W$, if $w$ has a reduced expression ending in $r$, then all reduced expressions for $w$ end in $r$. To see this, recall that any reduced expression for $w$ can be obtained from any other by use of braid relations, and there are no braid relations involving $r$, so there is no braid relation that will replace the $r$ at the end of $w$ with a different fundamental reflection. Moreover, if $w$ is an involution with a reduced expression ending in $r$, then since $w^{-1} = w$, every reduced expression for $w$ also begins with $r$. Let $x$ and $y$ be involutions and let $s\in \R\setminus\{r\}$. If both $x$ and $y$ have reduced expressions ending with $r$, then these reduced expressions also begin with $r$, and all reduced expressions for $x$ and for $y$ both begin and end with $r$. Therefore, both $x$ and $y$ are adjacent to $s$ and are thus distance at most 2 apart. Meanwhile, if neither $x$ nor $y$ have a reduced expression beginning with $r$, then both $x$ and $y$ are adjacent to $r$. If $x$ has a reduced expression ending in $r$ and $y$ does not (or vice versa), then $\ell(xy) = \ell(x) + \ell(y)$, because in order for any cancellation to occur, at least one braid relation involving $r$ would have to be used, in order to juxtapose and cancel any pairs of fundamental reflections from the expressions for $x$ and $y$. Therefore $x$ and $y$ are adjacent. Hence, $\diam(\Eh(W)) = 2$. 
	
	The final case  to consider is where $W$ has infinite rank. Since any element of $W$ has finite length, any reduced expression involves at most finitely many fundamental reflections. Thus, for any involutions $x$ and $y$ there is a fundamental reflection $r$ that does not appear in any reduced expression for $x$ or $y$. Clearly, $x$ and $y$ are both adjacent to $r$ in $\Eh(W)$, and so, again, $\diam(\Eh(W)) = 2$.
\end{proof}

Recall that an irreducible Coxeter group $W$ is compact hyperbolic if it is neither finite nor affine, but all proper parabolic subgroups of $W$ are finite. 

\begin{thm}\label{thm:finitediameter} Let $W$ be a Coxeter group of rank at least 2.
	\begin{enumerate}
		\item[(i)] If $W$ be a finite Coxeter group, then
	\[\diam(\Eh(W)) = \begin{cases} 
	1 & \text{ if } W=W(\A{2}) \text{ or } W(\A{1}\times \A{1}), \\
	3 & \text{ otherwise.}
	\end{cases}\]
		\item[(ii)] If $W$ is an affine Coxeter group, then
		\[\diam(\Eh(W)) = \begin{cases} 2 & \text{ if } W=W(\widetilde{A}_1),\\
	3 & \text{ otherwise.}
	\end{cases}\]
		\item[(iii)] If $W$ is compact hyperbolic, then $\diam(\Eh(W))=3$. 
	\end{enumerate}
\end{thm}
\begin{proof} 
Suppose that $W$ is finite, affine or compact hyperbolic. If $\rank(W) \geq 3$, then $\diam(\Eh(W))=3$ by Lemma~\ref{lem:maxparafinite}. Hence we may assume that $\rank(W) =2$.
\begin{enumerate}
\item[(i)] 
	Let $W$ be finite of rank 2 with generators $r$ and $s$ such that $o(rs) = m$. If $m=2$, then $W$ is of type $\A{1} \times \A{1}$, and if $m=3$, then $W$ is of type $\A{2}$. In both cases, $\I(W)\setminus\{w_0\} = \{r,s\}$, and $r$ is adjacent to $s$ in $\Eh(W)$. Thus, the diameter is 1. 
	
	Suppose $m > 3$. If $m$ is even, then $w_0$ is central and so $w_0r$ is an involution. Since $N(w_0r) = \Phi^{+}\setminus \{\alpha_r\}$, it follows that $w_0r$ is adjacent only to $r$ in $\Eh(W)$. Similarly, $w_0s$ is adjacent only to $s$. Therefore $w_0r$ and $w_0s$ share no common neighbours and $\diam(\Eh(W)) = 3$. If $m$ is odd, then every involution in $W$ is a reflection. Since $\ell((rs)^{(m-3)/2}r) = m-2$, any involution adjacent to $(rs)^{(m-3)/2}r$ in $\Eh(W)$ has length at most 2. But all reflections have odd length. Thus, $(rs)^{(m-3)/2}r$ is adjacent only to $s$, and similarly $(sr)^{(m-3)/2}s$ is adjacent only to $r$. As before, this implies that $\diam(\Eh(W)) = 3$. 
	\item[(ii)] The only affine Coxeter group with rank less than 3 is $W(\widetilde{A}_1)$, which has Dynkin diagram
	\begin{picture}(20.00,15)(-4,-2)
		\unitlength 1.00mm
		\linethickness{0.4pt}
		\put(6,2){\makebox(0,0)[cc]{$\infty$}}
		\put(0,0){\line(1,0){12}}
		\put(0,0){\circle*{1.5}}
		\put(12,0){\circle*{1.5}}
	\end{picture}\hspace*{0.7cm}. By Lemma~\ref{infinitediameter}, $\diam(\E(W(\widetilde{A}_1))) = 2$.
	\item[(iii)] All compact hyperbolic Coxeter groups have rank at least 3. \qedhere
\end{enumerate} \end{proof}	

\begin{lemma}\label{lem:infinitered}
	Suppose that $W$ is infinite and reducible with $W=W_1\times\cdots\times W_m$, where each $W_i$ is irreducible and $m>1$.		
	\begin{enumerate}[label=\rm{(\roman*)}]	
		\item If $\diam(\Eh(W_i))=2$ for all $i\in \{1,\ldots,m\}$, then  $\diam(\Eh(W))=2$.
		\item If $\diam(\Eh(W_i))=3$ for all $i\in \{1,\ldots,m\}$, then  $\diam(\Eh(W))=3$.
	\end{enumerate}
\end{lemma}
\begin{proof}
	\begin{enumerate}[label=\rm{(\roman*)}]
		\item By Theorem~\ref{thm:finitediameter}(i), each $W_i$ is infinite, and so $\Eh(W_i)=\E(W_i)$ by Theorem~\ref{thm:diam}. Let $x=(x_1,\ldots,x_m),y=(y_1,\ldots,y_m) \in \I(W)$. If $x_i$ is adjacent to $y_i$ in $\E(W_i)$ for all $i\in \{1,\ldots,m\}$, then $x$ is adjacent to $y$ in $\E(W)$. 
		Hence assume that there exists $i\in \{1,\ldots,m\}$ such that $x_i$ and $y_i$ are not adjacent in $\E(W_i)$. Since $\diam(\E(W_i))=2$, there exists $z_i \in \I(W_i)$ such that $x_i$ and $y_i$ are both adjacent to $z_i$. Hence $x$ and $y$ are both adjacent to the element of $\I(W)$ with $z_i$ in its $i^{\text{th}}$ position and the identity elsewhere.
		
		\item For all $i \in \{1,\ldots,m\}$ there exist $x_i,y_i \in \I(W_i)$ such that $d(x_i,y_i)=3$. Let $x:=(x_1,\ldots,x_m),y:=(y_1,\ldots,y_m) \in \I(W)$. Then clearly $x$ and $y$ are not adjacent. Suppose there exists $z=(z_1,\ldots,z_m) \in \I(W)$ adjacent to both $x$ and $y$. Since $\ell(x_iz_i)\leq \ell(x_i)+\ell(z_i)$, it follows that $\ell(xz)=\ell(x)+\ell(z)$ if and only if $\ell(x_iz_i)= \ell(x_i)+\ell(z_i)$ for all $i\in \{1,\ldots,m\}$. In order for $z$ to be an involution, either $z_i \in \I(W_i)$ or $z_i = 1$. The first is impossible because $d(x_i, y_i) = 3$. Therefore $z_i = 1$ for all $i$, and so $z=1 \notin \I(W)$, a contradiction. Thus $\diam(\Eh(W)) = 3$. \qedhere
	\end{enumerate}
\end{proof}

We note that Lemma~\ref{lem:infinitered} is best possible without knowing more about the irreducible components than just their diameters. For example, suppose there exists $i \in \{1,\ldots, m\}$ with $\diam(\Eh(W_i))=3$. Then $\diam(\Eh(W))=2$ if and only if there exists $j \in \{1,\ldots,m\}$ with both $\diam(\Eh(W_j))=2$ and the additional property that every pair of adjacent involutions in $\Eh(W_j)$ has at least one neighbour in common.
As the final lemma in this section shows, this additional property cannot be guaranteed for all $W_j$ with $\diam(\Eh(W_j))=2$. 

\begin{lemma}\label{infiniteredred}
	For $n \geq 2$, let $W_n$ be the Coxeter group given by $W_n = \langle r_1, \ldots, r_n \; | \; r_1^2 = \cdots = r_n^2=1 \rangle$. Then $\diam(\Eh(W_n)) = 2$. Furthermore, if $n=2$, then no pair of adjacent involutions in $\Eh(W_n)$ has a neighbour in common. If $n>2$, then every pair of adjacent involutions has a neighbour in common.
\end{lemma} 
\begin{proof}
	Every element in $W_n$ has a unique reduced expression, as there are no braid relations. For each $i$ let $X_i$ be the set of involutions in $W_n$ whose reduced expression ends in $r_i$. Then $\I(W_n)$ is the disjoint union $X_1\cup \cdots \cup X_n$ and for all $x$ in $X_i$ we have $\Delta_1(x) = \I(W_n)\setminus X_i$. Thus, involutions $x$ and $y$ are adjacent if and only if $x \in X_i$ and $y\in X_j$ for some $i\neq j$, and the set of mutual neighbours of $x$ and $y$ is $\I(W_n) \setminus (X_i\cup X_j)$. The result now follows immediately.  
\end{proof}

\section{Valencies in $\E(W)$}\label{sec:val}

We begin with the following well-known result.

\begin{lemma}\label{lendown} Suppose that $W$ is a finite rank Coxeter group. Let $r$ be a fundamental reflection and
	$x$  an involution of $W$, and assume that $r$ and $x$ do not commute.  If $x\cdot \alpha_r  \in \Phi^-$, then $\ell(rx) = \ell(xr) = \ell(x) - 1$ and $\ell(rxr) = \ell(x) - 2$. If $x\cdot \alpha_r  \in \Phi^+$, then $\ell(rx) = \ell(xr) = \ell(x) + 1$ and $\ell(rxr) = \ell(x) + 2$.
\end{lemma} 	
\begin{proof}
	Note that $(xr)^{-1} = r^{-1}x^{-1} = rx$, and so $\ell(xr) = \ell(rx)$.  Assume that $x\cdot \alpha_r  \in \Phi^-$. Then  $\ell(rx) = \ell(xr) = \ell(x) - 1$ by Lemma~\ref{lem:pos/neg}. Now, $\ell(rxr) = \ell(rx) - 1$ if and only if $rx\cdot \alpha_r  \in \Phi^-$. But $rx\cdot \alpha_r = r\cdot (x\cdot \alpha_r)$. Since $\ell(xr) < \ell(x)$, we have that $x\cdot \alpha_r  \in \Phi^-$, so the only way that $r\cdot (x\cdot \alpha_r) \in \Phi^+$ is if $x\cdot \alpha_r = -\alpha_r$. But if that happens, then $x$ commutes with $r$, contrary to our hypothesis. Hence $\ell(rxr) = \ell(x) - 2$. Similarly, if $x\cdot \alpha_r \in \Phi^+$, then the only way that $r\cdot  (x\cdot \alpha_r) \in \Phi^-$ is if $x\cdot \alpha_r = \alpha_r$, which again implies that $x$ commutes with $r$. So $\ell(rxr) = \ell(x) + 2$.
\end{proof}

\begin{lemma}\label{cardinality}
	Suppose that $W$ is a Coxeter group of finite rank, let $r$ be a fundamental reflection of $W$ and write $\overline{\I(W)} = \I(W)\setminus \{r\}$. Then $|\Delta_1(r)|  = |\overline{\I(W)}\setminus \Delta_1(r)|$. In particular, if $W$ is finite, then $|\Delta_1(r)| = \frac{1}{2}(|\I(W)| - 1)$. 
\end{lemma}

\begin{proof}
	Let $r $ be a fundamental reflection of $W$. Then $\overline{\I(W)}$ is the disjoint union of two sets: $X = \{x\in \overline{\I(W)} \; | \;  xr = rx\}$ and $Y = \{y\in \overline{\I(W)} \; | \; yr \neq ry\}$. 
	Now, $x \in X$ if and only if $xr\in X$, and $\ell(xr) = \ell(x) + 1$ if and only if $\ell((xr)r) = \ell(x) = \ell(xr) - 1$. Therefore, $X \cap \Delta_1(r)$ and $X \cap (\overline{\I(W)} \setminus \Delta_1(r))$ have the same cardinalities. Similarly, $y\in Y$ if and only if $ryr\in Y$. Moreover, if $\ell(yr) = \ell(y) + 1$, then $\ell(ryr) = \ell(y) + 2$ by Lemma~\ref{lendown}.  So $\ell(r(ryr)) = \ell(yr) = \ell(ryr) - 1$. On the other hand, if $\ell(yr) = \ell(y) - 1$, then, again by Lemma~\ref{lendown}, $\ell(ryr) = \ell(y) - 2$ and so $\ell(r(ryr)) = \ell(ryr) + 1$. Thus $Y \cap \Delta_1(r)$ and $Y \cap (\overline{\I(W)} \setminus \Delta_1(r))$ also have the same cardinality, so proving the lemma.
	\end{proof}

Our next result and its corollary show that when $W$ is finite no involution can have more neighbours than a fundamental reflection. 

\begin{lemma}\label{gregarious} Let $x$ be a vertex of $\Eh(W)$ and suppose $r\in R$ with $l(xr)<l(x)$. Then $\Delta_1(x) \subseteq \Delta_1(r)$. If $W$ is finite, then $\Delta_1(x) = \Delta_1(r)$ if and only if $x=r$.  
\end{lemma} 
\begin{proof}
	Let $y\in \Delta_1(x)$. Then $N(y) \cap N(x) = \emptyset$. But $N(r) = \{\alpha_r\}\subseteq N(x)$, and hence $N(y) \cap N(r) = \emptyset$. Therefore, $\Delta_1(x) \subseteq \Delta_1(r)$. Now, suppose $W$ is finite. If $r$ commutes with $w_0$, then $rw_0 \in \Delta_1(r)$. But $N(rw_0) = \Phi^{+}\setminus \{\alpha_r\}$. Thus the only way for $N(x) \cap N(rw_0) = \emptyset$ is for $x=r$. Hence $\Delta_1(x) = \Delta_1(r)$ implies $x=r$. If $r$ does not commute with $w_0$, then $rw_0r\in \Delta_1(r)$ and by Lemma~\ref{lendown}, $\ell(rw_0r) = \ell(w_0) - 2$. Hence if $rw_0r\in \Delta_1(x)$, we have $\ell(x) \leq 2$. If $\ell(x) = 1$, then clearly $x=r$. If $\ell(x) = 2$, then (since $x$ is an involution) $x = rs$ where $s\in R$ and $sr = rs$. But then $s\in \Delta_1(r)\setminus \Delta_1(x)$. In all cases, we see that $\Delta_1(x) = \Delta_1(r)$ if and only if $x=r$.
\end{proof}

\begin{cor}
	If $W$ is finite and $x\in \E(W)$, then $|\Delta_1(x)| \leq \frac{1}{2}\left(|\I(W)| - 1\right)$, with equality if and only if $x\in R$. 
\end{cor}

We note that when $W$ is infinite, it is possible for $\Delta_1(x) = \Delta_1(r)$ when $x \neq r$. For example, in the infinite dihedral group $W = W(\widetilde A_1)$, generated by $r$ and $s$ with $m_{rs} = \infty$, let $X = \{r, rsr, rsrsr, \ldots\}$ and $Y = \{s, srs, srsrs\ldots\}$. Then $\I(W) = X\cup Y$. For all $x\in X$ we have $\Delta_1(x) = Y = \Delta_1(s)$ and for all $y \in Y$ we have $\Delta_1(y) = X = \Delta_1(r)$. \medskip

We next restrict our attention to $\Sym(n)$ and determine formulae for the number of neighbours of minimal length involutions in their conjugacy classes. The conjugacy classes of involutions in $\Sym(n)$ are parametrized by the number of transpositions in the involution when written as a product of disjoint cycles. 

\begin{lemma}\label{wlog}
	Suppose $x$ and $y$ are conjugate involutions both of minimal length in their conjugacy class in $\Sym(n)$. Then $|\Delta_1(x)| = |\Delta_1(y)|$.
\end{lemma}
\begin{proof}
 It is sufficient to prove that whenever $x$ is a product of $m$ mutually commuting distinct fundamental reflections (transpositions of the form $(a_i,a_i + 1)$) in $\Sym(n)$, then $|\Delta_1(x)| = |\Delta_1((12)(34)\cdots (2m-1,2m))|$. Suppose that $x\neq (12)(34) \cdots (2m-1,2m)$. Then there is some $a$ with $1<a<n$ such that $x$ contains the transposition $r=(a, a+1)$ and fixes $a-1$. Let $y$ be $x$ with $r$ replaced by $s = (a-1, a)$. We will show that $|\Delta_1(y)| = |\Delta_1(x)|$. 
	First, note that $\Delta_1(x)$ is the disjoint union of the following three sets.
	\begin{align*}
		U_1 &= \{z\in \Delta_1(x) \; | {\; \text{$z$ fixes $a-1$, $a$, and $a+1$}}\}\\
		U_2 &= \{z\in \Delta_1(x) \; | \; {\text{$z$ contains $(a-1, a)$ and fixes $a+1$}}\}\\
		U_3 &= \Delta_1(x)\setminus (U_1\cup U_2).
		\end{align*} 
	Meanwhile, $\Delta_1(y)$ is the disjoint union of the following sets.
	\begin{align*}
	V_1 &= \{z\in \Delta_1(y) \; | \; {\text{$z$ fixes $a-1$, $a$, and $a+1$}}\}\\
	V_2 &= \{z\in \Delta_1(y) \; | \; {\text{$z$ contains $(a, a+1)$ and fixes $a-1$}}\}\\
	V_3 &= \Delta_1(y)\setminus (V_1\cup V_2).
\end{align*} 
Clearly, $rU_1 = V_2$ and $sU_2 = V_1$. We claim that $rsU_3sr = V_3$. Suppose $z\in U_3$. Then $z\cdot \alpha_r \in \Phi^+$. We have $rszsr\cdot\alpha_s = rsz\cdot \alpha_r$. Since $z\cdot \alpha_r\in \Phi^+$, $rsz\cdot \alpha_r \in \Phi^-$ if and only if $z\cdot \alpha_r = \alpha_s$ or $z\cdot \alpha_r = \alpha_r + \alpha_s$. In the first case, since $\alpha_r = e_{a} - e_{a+1}$ and $\alpha_s = e_{a-1}-e_a$, we would have $z(a+1) = a$ and $z(a) = a-1$, contradicting the fact that $z$ is an involution. In the second case, since $\alpha_r + \alpha_s = e_{a-1}-e_{a+1}$, we have that $z$ contains $(a-1, a)$ and fixes $a+1$, contradicting the fact that $z\in U_3$. Therefore, $rszsr\cdot \alpha_s \in \Phi^+$, which implies $rszsr\in \Delta_1(y)$. If $rszsr \in V_1$, then $z$ commutes with both $r$ and $s$, meaning $rszsr = z$, and thus $z\in U_1$, a contradiction. If $rszsr\in V_2$, then $z$ contains $(a-1, a+1)$ and fixes $a$. But then $z\cdot \alpha_r \in \Phi^-$, contradicting the fact that $z\in \Delta_1(x)$. Hence $rszsr\in V_3$. Thus, $|\Delta_1(x)| = |\Delta_1(y)|$. We may repeat this process until we obtain an element with no further fixed points in the set $\{1, 2, \ldots, 2m\}$. Hence, $|\Delta_1(x)| = |\Delta_1((12)(34)\cdots (2m-1, 2m))|$. 
\end{proof}

Lemma~\ref{wlog} proves that $\delta(m,n)$, as described in Section~\ref{sec:intro}, is well-defined. We now prove Theorem~\ref{thm:valency}, which states that if $m\geq 2$ and $n\geq 2m$, then 
\[\delta(m,n) = \textstyle\frac{1}{2}\big(\delta(m-1,n) + (m-1)\delta(m-2,n-4) + m-2\big).\]

\begin{proof}[Proof of Theorem~\ref{thm:valency}.] Given an involution $x$, we find $x$ as a vertex of $\Eh(\Sym(n))$ whenever the support of $x$ is contained in $\{1, \ldots, n\}$. We write $\Delta(x,n)$ for $\Delta_1(x)$ in $\Eh(\Sym(n))$.
Let $y = (12)(34) \cdots (2m-3, 2m-2)$ and let $x = y(n-1, n)$. Then $\delta(m,n) = |\Delta(x,n)|$.
Observe that 
\[\Delta(x,n) = \Delta(y,n) \cap \Delta((n-1, n), n) = \{z\in \Delta(y,n)\;  | \; z(n-1)< z(n)\}.\]
Now, $\Delta(y,n)$ is the disjoint union of the following sets.
\begin{align*}
	U_1 &= \{z\in \Delta(y,n) \; | \; {\text{$z$ fixes $n$ and $n-1$}}\}\\
	U_2 &= \{z\in \Delta(y,n) \; | \; {\text{$z$ contains $(n-1, n)$}}\}\\
	U_3 &= \{z\in \Delta(y,n) \; | \; {\text{$z$ fixes exactly one of $n$ and $n-1$}}\}\\
	U_4 &= \{z\in \Delta(y,n) \; | \; {\text{$z$ contains $(a, n-1)(b, n)$ some $a, b < n-1$}}\}
\end{align*}
Clearly $|U_1| = \delta(m-1,n-2)$, $|U_2| = \delta(m-1,n-2) + 1 = |U_1| + 1$, and $U_1\subseteq \Delta(x,n)$ while $U_2\cap \Delta(x,n) = \emptyset$. In $U_3$, we can pair the elements into those with cycles $(a, n-1)(n)$ for some $a < n-1$ (and these elements are contained in $\Delta(x,n)$), and those with cycles $(a, n)(n-1)$ (and these elements are not contained in $\Delta(x,n)$). Therefore $|U_3\cap \Delta(x,n)| = \frac{1}{2}|U_3|$. Finally, consider $z$ in $U_4$, so that $z$ contains $(a,n-1)(b,n)$ for some $a, b < n-1$. In all cases except where $a = 2i-1$ and $b=2i$, for some $1 \leq i < m$, both $z$ and $z^{(n-1,n)}$, which is $z$ with $(a, n-1)(b, n)$ replaced with $(b, n-1)(a,n)$, are contained in $U_4$. Moreover exactly one of $z$ and $z^{(n-1,n)}$ lies in $\Delta(x,n)$. 	Let 
\[U_5 = \big\{z\in \Delta(y,n) \; \big| \; {\text{$z$ contains $(2i-1, n-1)(2i, n)$ some $i \in \{1, \ldots, m-1\}$}}\big\} \subseteq U_4.\] 
	Then $U_5\subseteq \Delta(x,n)$ and exactly half the elements of $U_4\setminus U_5$ lie in $\Delta(x,n)$. Moreover, for each $i$ in $\{1, \ldots, m-1\}$, there is a bijection between $\{z\in \Delta(y,n) \; | \; {\text{$z$ contains $(2i-1, n-1)(2i, n)$}}\}$ and $\Delta\left((12)(34)\cdots (2m-5,2m-4),n-4\right)$. Thus, $|U_5| = (m-1)(\delta(m-2,n-4)+1)$. Gathering all these observations together, we see that
	\begin{align*}
		\delta(m,n) &= |\Delta(x,n)| = |U_1| + \textstyle\frac{1}{2}|U_3| + \textstyle\frac{1}{2}|U_4\setminus U_5| + |U_5|\\
		&= \textstyle\frac{1}{2}\big(|U_1| + (|U_2|-1) + |U_3| + |U_4| - |U_5| + 2|U_5| 
		\big)\\
		&= \textstyle\frac{1}{2}\big(\delta(y,n) + |U_5| - 1\big) \\	
		&= \textstyle\frac{1}{2}\big(\delta(m-1,n) + (m-1)\delta(m-2,n-4) + m-2
		\big).\qedhere
	\end{align*}
\end{proof}
By repeated use of Theorem~\ref{thm:valency} and Lemma~\ref{cardinality}, we obtain the following corollary.  
\begin{cor} Let $n \geq 2$. Then \label{m-is-1234}
	\begin{align*}
\delta(1,n) &= \textstyle\frac{1}{2}\big(|\I(\Sym(n))| - 1\big)& \\		
\delta(2,n) &= \textstyle\frac{1}{4}\big(|\I(\Sym(n))| + 2|\I(\Sym(n-4))|- 1\big)& (n \geq 4)\\		
\delta(3,n) &= \textstyle\frac{1}{8}\big(|\I(\Sym(n))| + 6|\I(\Sym(n-4))|- 1\big)& (n \geq 6)\\		
\delta(4,n) &= \textstyle\frac{1}{16}\big(|\I(\Sym(n))| + 12|\I(\Sym(n-4))| + 12|\I(\Sym(n-8))| + 9\big)& (n \geq 8)
	\end{align*}
\end{cor}

\section{Pendant Elements}\label{sec:pendant}

In this section we determine the elements of $\I(W)$ with minimal valency in $\Eh(W)$. 
From now on, $W$ is assumed to be a finite irreducible Coxeter group of rank $n \geq 2$, with $R =\{r_1, \dots, r_n \}$ its set of fundamental reflections and, where relevant, has Dynkin diagrams as given in Section~\ref{sec:background}. 

As we have seen in Theorem~\ref{thm:diam}, if  $w \in \I(W) \backslash \{w_0\}$, then $|\Delta_1(w)|\geq 1$. Hence we are interested in those $w \in \I(W) \backslash \{w_0\}$ with $|\Delta_1(w)|=1$. We call such elements {\em pendant elements} and denote the set of such elements by $\P(W)$.
For $x \in W$,  we observe that $w_0x \in \I(W)$ if and only if $x^{w_0}=x^{-1}$. For ease, we shall often work with elements of the form $w_0x$, and begin with a preliminary lemma for elements of this form.
\begin{lemma}\label{lem:equiv} Suppose that $x \in W$ is such that $w_0x \in \I(W)$.  Let $w \in \I(W)$. Then $w \in \Delta_1(w_0x)$ if and only if $\ell(xw)=\ell(x)-\ell(w)$.
\end{lemma}
\begin{proof}
Firstly, as $w_0$ is the longest element, we have $\ell(w_0xw)=\ell(w_0)-\ell(xw)$. If $w \in \Delta_1(w_0x)$, then, in addition,
$\ell(w_0xw)=\ell(w_0x)+\ell(w)=\ell(w_0)-\ell(x)+\ell(w)$.
Hence, $\ell(xw)=\ell(x)-\ell(w)$.

On the other hand, if $\ell(xw)=\ell(x)-\ell(w)$, then 
\[\ell(w_0x)+\ell(w)=\ell(w_0)-\ell(x)+\ell(w)=\ell(w_0)-\ell(xw)=\ell(w_0xw), \] 
and hence $w \in \Delta_1(w_0x)$.
\end{proof}

An immediate consequence of Lemma~\ref{lem:equiv} is our next result.

\begin{lemma}\label{lem:w0r} Let $r \in R$. If $w_0r \in \I(W)$, then $\Delta_1(w_0r)=\{r\}$.
\end{lemma}

\begin{lemma}\label{lem:pendantw0central}  Suppose that $w_0\in \ZZ(W)$. (This occurs when $W$ is of type $\B{n}$, type $\D{n}$ with $n$ even, type $\F{4}$,  $\H{3}$, $\H{4}$, $\mathrm{E}_7$, or $\mathrm{E}_8$, or type $\mathrm{I}_2(m)$ with $m$ even.) Then $\P(W) = \{ w_0r \; | \; r \in R \}$, with $\Delta_1(w_0r)=\{r\}$ for all $r \in R$. 
\end{lemma}
\begin{proof}
Since $w_0 \in \ZZ(W)$, we have $w_0r \in \I(W)$. By Lemma~\ref{lem:w0r},  $\Delta_1(w_0r)=\{r\}$ and so $w_0r \in \P(W)$.
Let $w \in \I(W) \backslash (\{w_0\} \cup w_0R)$. Then $ww_0 \in \I(W) \backslash R$. From
\[\ell(w)+\ell(ww_0)=\ell(w)+\ell(w_0)-\ell(w) = \ell(w_0) = \ell(www_0)\]
it follows using Lemma~\ref{lem:equiv} that $ww_0 \in \Delta_1(w)$. By Lemma~\ref{lem:fundneighbour}, $w$ has a neighbour in $R$ and so $|\Delta_1(w)| \geq 2$.  Hence $w \notin \P(W)$ and the lemma holds.
\end{proof}

The remaining finite irreducible Coxeter groups to consider are those for which $w_0$ is non-central. These are type $\A{n}$ with $n \geq 2$, type $\D{n}$ with $n$ odd, type $\mathrm{E}_6$, and type $\mathrm{I}_2(m)$ for $m$ odd and $m\geq 5$. The next lemma covers this last case (as well as giving greater detail about $\mathrm{I}_2(m)$ for arbitrary $m$).

\begin{lemma}\label{lem:I2m} Let $W =  \mathrm{I}_2(m)$, with $m \geq 5$. The valency distribution of $\E(W)$ is 
\[0^1\cdot 1^2\cdot 2^2\cdots \left\lfloor\textstyle\frac{m}{2}\right\rfloor^2.\]
If $m$ is even, then $\P(W) = \{ w_0r_1, w_0r_2\}$, with $\Delta_1(w_0r_1)=\{r_2\}$ and $\Delta_1(w_0r_2)=\{r_1\}$. If $m$ is odd, then $\P(W) = \{w_0r_1r_2, w_0r_2r_1\}$, with $\Delta_1(w_0r_1r_2)=\{r_2\}$ and $\Delta_1(w_0r_2r_1)=\{r_1\}$ .
\end{lemma}
\begin{proof}
	The non-central involutions in $W$ are $(r_1r_2)^ir_1$ and $(r_2r_1)^ir_2$, for $0\leq i\leq \lfloor\frac{m-2}{2}\rfloor$. We have $\Delta_1((r_1r_2)^ir_1) = \{(r_2r_1)^jr_2 \; | \; i+j\leq \lfloor\frac{m-2}{2}\rfloor\}$ and hence $|\Delta_1((r_1r_2)^ir_1)|= \lfloor\frac{m}{2}\rfloor -i$. Similarly,  $\Delta_1((r_2r_1)^ir_2) = \{(r_1r_2)^jr_1 \; | \; i+j\leq \lfloor\frac{m-2}{2}\rfloor\}$, and hence $|\Delta_1((r_2r_1)^ir_2)|= \lfloor\frac{m}{2}\rfloor -i$. Noting that $w_0 = (r_1r_2)^{m/2}$ when $m$ is even and $w_0 = (r_1r_2)^{(m-1)/2}r_1=(r_2r_1)^{(m-1)/2}r_2$ when $m$ is odd, the result now follows. \end{proof}
The following may be verified using \textsc{Magma}\cite{magma}.
\begin{lemma}\label{E6} If $W$ is of type $\mathrm{E}_6$, then  \[\P(W) = \{w_0r_2, w_0r_4, w_0r_5r_4r_3, w_0r_3r_4r_5, w_0r_6r_5r_4r_3r_1,  w_0r_1r_3r_4r_5r_6 \}\] with 
$\Delta_1(w_0\cdots r_i)=\{r_i\}$.
\end{lemma}

\subsection{$W$ of type $\A{n}$}

In this subsection we assume that $W$ is of type $\A{n}$ and $n\geq 2$.  Recall that $R = \{r_1, \dots ,r_n \}$. Using the notation defined in Section~\ref{sec:background}, we write $[i_1,i_2, \dots , i_{j-1},i_j]$ for  \[r_{i_1}r_{i_2} \cdots r_{i_{j-1}}r_{i_j}\] where  $\{i_1, \ldots, i_j\}\subseteq \{1, \ldots, n\}$, and when $i \leq j$ we set 
$[ i \nearrow j]:=r_ir_{i+1} \cdots r_{j-1}r_j$ and $ [j  \searrow i]:=r_jr_{j-1} \cdots r_{i+1}r_i$. 

Let \[X=\left\{[i \nearrow (n+1-i)],  \; [(n+1-i) \searrow i]  \; \Big| \; i \in \left\{1,\ldots, \left\lceil \textstyle\frac{n}{2} \right\rceil \right\}\right\}.\]

The following is the main result of this subsection.

\begin{lemma}\label{lem:pendantAn} The pendant elements in $\E(W)$ are $w_0x$ for $x \in X$, with
\[\Delta_1(w_0[ i \nearrow (n+1-i)])=\{r_{n+1-i}\} \quad \text{and} \quad \Delta_1(w_0 [(n+1-i) \searrow i])=\{r_i\}\]
for $i \in \{1,2,\ldots, \lceil \frac{n}{2} \rceil\}$.
\end{lemma}

The proof of Lemma~\ref{lem:pendantAn} requires two preparatory lemmas about sequential elements. We call an element $x$ of $W$ a \emph{sequential element} if for some $r_a \in R$ and $\lambda,\mu \geq 0$ it has a reduced expression of the form
\[x = [a-\mu \nearrow a-1, a+{\lambda}\searrow a].\]

Because $\lambda$ and $\mu$ can both be zero, it is clear that every reduced expression for any non-identity element of $W$ ends with a sequential element. 

\begin{lemma}\label{lem:updown} Let $x \in W$. Then either 
\begin{enumerate}[label=\rm{(\roman*)}]
\item $x$ is a sequential element of $W$; or
\item there exist distinct $y_1,y_2 \in \I(W)$ with $\ell(xy_j)=\ell(x)-\ell(y_j)$ for $j \in \{1,2\}$.
\end{enumerate}
\end{lemma}

\begin{proof} Suppose that $x$ is not a sequential element of $W$. Fix a reduced expression for $x$ and let $r_a = [a]$ be the final fundamental reflection in the expression. Since $[a +i]$ and $[a-j]$ commute for $i ,j \geq 1$, it follows that there exist integers $\mu, \lambda \geq 0$ and $b \notin \{a+\lambda+1,a-\mu-1\}$ such that $x$ has a reduced expression of the following form. 
   \[x = [\dots, b, a-\mu \nearrow a-1, a+{\lambda} \searrow a] \]
We now construct $y_1$ and $y_2$ based on the possibilities for $b$. 
Since this is a reduced expression for $x$, it follows that the cases of $b=a-\mu <a$ and $b=a+\lambda>a$ do not occur. 

If $b > a+\lambda+1$ or $b < a-\mu-1$, then the result follows with $y_1=[a]$ and $y_2=[b]$. 
Meanwhile, if $b=a+i$ with $i \in \{1,\ldots, \lambda-1\}$, then take $y_1=[a]$ and $y_2=[a+i+1]$. If $b=a-j$ with $j \in \{1,\ldots, \mu-1\}$, then let $y_1=[a]$ and $y_2=[a-j-1]$.

Finally assume that $b=a$. Since the expression for $x$ is reduced it follows that either $\mu \neq 0$ or $\lambda \neq 0$. Hence there exists a reduced expression for $x$ ending in
\[y_2=\begin{cases} 
[a,a-1,a+1,a] & \text{if }\lambda, \mu \neq 0\\
[a,a-1,a] & \text{if } \lambda=0\\
[a,a+1,a] & \text{if } \mu =0.\\
\end{cases}\]
The result then follows with $y_1=[a]$ and $y_2$ as above.
\end{proof}
Now, for $x$ a sequential element, we determine when $w_0x \in \I(W)$. 
\begin{lemma}\label{lem:cont} Let $x \in W$ be a sequential element. Then $w_0x \in \I(W)$ if and only if $x \in X$.
\end{lemma}
\begin{proof}
Recall that for $W$ of type $\A{n}$ conjugation by $w_0$ interchanges $r_i$ and $r_{n+1-i}$ for $i \in \{1,\ldots, \lfloor \frac{n}{2} \rfloor\}$.  Consequently for $x \in X$ we have $x^{w_0}=x^{-1}$, and so $w_0x \in \I(W)$.

Let $x \in W$ be a sequential element such that $w_0x \in \I(W)$. Then there exists $a \in \{1,\ldots,n\}$ and $\mu,\lambda \geq 0$ such that  $x=[a-\mu \nearrow a-1, a+\lambda \searrow a]$.
We calculate the possibilities for $\mu$ and $\lambda$. For $w \in W$, let $J(w)$ be the set containing $i \in \{1,\ldots,n\}$ exactly when
$[i]$ is in a reduced expression for $w$. Then $J(x)=J(x^{-1})=J(x^{w_0})$. From $\max J(x)=a+\lambda$ and $\max J(x^{w_0})=n+1-a+\mu$, we deduce that $\lambda-\mu=n+1-2a$.

Every reduced expression for $x$ begins with either $[a+\lambda]$ or $[a-\mu]$, and so every reduced expression for $x^{w_0}$ begins with either $[n+1-a-\lambda]$ or $[n+1-a+\mu]$. Also every reduced expression for $x$ ends in $[a]$, and so a reduced expression for $x^{-1}$ begins with $[a]$. Hence $n+1-2a \in \{\lambda, -\mu\}$.  If $n+1-2a=\lambda$, then $\mu=0$ and so $x=[(n+1-a) \searrow a]$. If $n+1-2a=-\mu$, then $\lambda=0$ and so $x=[(a-\mu) \nearrow (n+1-(a-\mu))]$. Therefore $x \in X$, so proving the lemma.
\end{proof}

We can now prove the main result of this subsection.
\begin{proof}[Proof of Lemma~\ref{lem:pendantAn}]
Let $x \in W$ with $w_0x \in \I(W)$.  
Then by Lemmas~\ref{lem:updown} and \ref{lem:cont} either $x \in X$ or there exist distinct elements $y_1,y_2 \in \Delta_1(w_0x)$. Therefore if $x \notin X$, then $w_0x$ is not pendant.

Let $x=[(n+1-i) \searrow i]$ and $w \in \Delta_1(w_0x)$. Then $\ell(xw)=\ell(x)-\ell(w)$ by Lemma~\ref{lem:equiv}. Therefore $w \in \{ 1
,[i], [i,(i+1)],\ldots, [i \nearrow (n+1-i)]\}$, and since $w \in \I(W)$ it follows that $w=[i]$. The argument for $x = [i \nearrow (n+1-i)]$ is similar. 
\end{proof}
\subsection{$W$ of type $\D{n}$ with $n$ odd} \label{subsec:pendantDn}

In this subsection we suppose that $W$ is of type $\D{n}$ with $n$ odd, and fix the following set of elements of $W$
\[L=\left\{[1],[2],\ldots,[n-2],[n,n-2,n-1], [n-1,n-2,n]\right\}.\]
The next lemma is the main result of this subsection.

\begin{lemma}\label{lem:pendantDn} Let $n \geq 5$. Then the pendant elements in $\E(W)$ are $w_0 y$ for $y \in L$, with
$\Delta_1(w_0 [i])=\{[i]\}$, $\Delta_1(w_0 [n, n-2, n-1])=\{[n-1]\}$ and  $\Delta_1(w_0 [n-1, n-2,n])=\{[n]\}$
for $i \in \{1,\ldots, n-2\}$.
\end{lemma}

Recall that conjugating by $w_0$ interchanges $[n-1]$ and $[n]$, and fixes all the other $[i]$. 
Let $J=\{1,\ldots,n-1\}$. Then $W_J$ is a parabolic subgroup of $W$ of type $\A{n-1}$. For the remainder of this section fix $x \in W \backslash L$ such that $w_0x \in \I(W)$. By Lemma~\ref{lem:cosetreps}, there exists $u \in$ $_JX$ and $v \in W_J$ such that $x=uv$ with $\ell(x)=\ell(u)+\ell(v)$. We begin with a preliminary lemma. We will use the shorthand $[\dots, i_1, i_2, \ldots, i_j]$ to mean a reduced expression ending with $r_{i_1}r_{i_2}\cdots r_{i_j}$.

\begin{lemma}\label{lem:uv} \begin{enumerate}[label=\rm{(\roman*)}]
\item If $\ell(u) \geq 1$, then $u =[ \dots, n]$. If $\ell(u) \geq 2$, then $u = [ \dots, n-2, n]$.
\item Suppose that $x = [r]z$ with $r \in R$ and $z \in W$, is a reduced expression for $x$. Then $x$ also has a reduced expression $(z^{w_0})^{-1}[r^{w_0}]$.
\item If $v \in W_{J \setminus \{n-1\}}$, then $u\not= [n]$.
\item If $u = 1$, then $x = v \in \I(W_{J\setminus \{n-1\}}).$
\end{enumerate}
\end{lemma}

\begin{proof} Part (i) follows from Algorithm B described in \cite[\S 2.1]{GeckPfeiffer}. For (ii) we have $[r^{w_0}]z^{w_0} = x^{w_0} = x^{-1}$ and so $x = (z^{w_0})^{-1}[r^{w_0}]$.

Suppose $u = [n]$, and so $x = [n]v$. Then, as $v \in W_{J \setminus \{n-1\}}$, by (ii) $x = v^{-1}[n-1] \in W_J$, forcing $[n] \in W_J$, a contradiction. So  (iii) holds.

Assume that $x = v \not\in W_{J\setminus \{n-1\}}$. Then we may write $x = c[n-1]e$ where $c,e \in W_{J\setminus \{n-1\}}$ (see \cite[Example 2.2.4]{GeckPfeiffer}).  Now $x^{-1} = x^{w_0} = c[n]e,$ whence, as $x^{-1} \in W_J$, $[n] \in W_J$, a contradiction. Thus $x = v \in W_{J\setminus \{n-1\}}$ and then $x^{-1} = x^{w_0} = x$ means $x \in \I(W_{J\setminus \{n-1\}})$.
\end{proof}

\begin{lemma}\label{xequ} Suppose that $v = 1$ (so $x = u$).  Then $x$ has a reduced expression of the form $b[n,n-2,n-3,n-1,n-2,n]$ for some $b \in W$. In particular, $[n],  [n-3,n-1]^{[n-2,n]} \in \Delta_1(w_0x)$.
\end{lemma}

\begin{proof} The analogue of Theorem~\ref{Dncosetreps} for distinguished left coset representatives gives that $x=ba$ where either (i)  $a \in \{ [n], [n - 1,n-2, n]\}$ and $b \in W_{\{1,\ldots,n-2\}}$; or (ii) $a=[n,n-2,n-3,n-1,n-2,n]$. Suppose for a contradiction that (i) holds. 

By assumption $x \notin L \supseteq \{[n],[n-1,n-2,n]\}$, and so $b\neq 1$. From $w_0x \in \I(W)$ it follows that $x^{w_0}=x^{-1}$, and consequently by Lemma~\ref{lem:cosetreps}(ii), $x^{w_0}$ is a distinguished right coset representative for $W_J$ in $W$. However $b \in W_{R \backslash \{n-1,n\}}$, and so $x^{w_0} = b^{w_0}a^{w_0} = ba^{w_0}$ cannot be such a distinguished right coset representative. Therefore, alternative (ii) must hold, which proves the lemma.
\end{proof}

\begin{lemma} \label{2orseq} Either $|\Delta_1(w_0x)|\geq 2$ or $v$ is a sequential element of $W_J$ and $\ell(u) \geq 1$ (so $u = [ \dots , n]$).
\end{lemma}

\begin{proof}
If $v=1$, then $|\Delta_1(w_0x)|\geq 2$ by Lemma~\ref{xequ}. If $u=1$, then $v \in \I(W_{J \setminus \{n-1\}})$ by Lemma~\ref{lem:uv}(iv). In particular, $v \neq [n-1],[n]$ which combined with $v=x \notin L$ implies that $v  \in \I(W) \backslash R$. By Lemma~\ref{lem:ex} there exists $r \in R$ such that $\ell(vr) = \ell(v) -1$. Thus by Lemma~\ref{lem:equiv} $v$ and $r$ are distinct elements of $\Delta_1(w_0x)$. Hence we may assume that $u,v \neq 1$, and so the result follows by Lemmas~\ref{lem:equiv} and \ref{lem:updown}.
\end{proof}

\begin{lemma}\label{lem:Dnatleast2}  If $x \not \in L$ and $w_0x  \in \I(W)$, then $|\Delta_1(w_0x)| \geq 2$.
\end{lemma}
\begin{proof} By Lemma~\ref{2orseq} we may assume that $v$ is a sequential element of $W_J$ and $u = [ \dots,n]$. So for some $a \in \{1,\ldots,n-1\}$ and $\lambda,\mu \geq 0$ we have
\[x = [\dots,n,a-\mu \nearrow a-1, a+{\lambda} \searrow a].\]

We split into four cases: (i) $v \in W_{R \backslash \{n-1\}}$; 
(ii) $a=n-2$; (iii) $a=n-1$ and $\mu \geq 1$; and (iv) $a \leq n-3$ and $a+\lambda \in \{n-1,n-2\}$.
 
In case (i), the elements $v$ and $[n]$ commute,  and hence $[a],[n] \in \Delta_1(w_0x)$. 

Suppose case (ii) holds. Then $\lambda \in \{0,1\}$. If $\lambda=0$, then $u \neq [n]$ by Lemma~\ref{lem:uv}(iii). If $\lambda=1$, then $u\neq [n]$ since it can be checked that $[n]v([n]v)^{w_0}\neq 1$. Therefore $u = [\dots, n-2,n]$ by Lemma~\ref{lem:uv}(i). The result then follows with $\Delta_1(w_0x)$ containing $[n-2]$ and one of the following.
\[\begin{cases}
[n]^{[n-2]} & \text{ if } \lambda=0,\mu=0\\
[n, n-3]^{[n-2]} & \text{ if } \lambda=0,\mu\geq 1\\
[n, n-1]^{[n-2]} & \text{ if } \lambda=1,\mu=0\\
[n,n-3,n-1]^{[n-2]} & \text{ if } \lambda=1,\mu\geq 1
\end{cases}\]

Next assume that (iii) holds. If $\mu=1$, then $v=[n-2,n-1]$. Since $x \notin L \supseteq \{[n,n-2,n-1]\}$, Lemma~\ref{lem:uv}(i) implies that 
\[x=[\ldots,n-2,n][n-2,n-1]=[\ldots,n,n-2,n,n-1].\]
Hence $[n-1],[n] \in \Delta_1(w_0x)$. If $\mu \geq 2$, then by Lemma~\ref{Dncosetreps} applied to $u^{-1}$, either $u=[\ldots,n-1,n-2,n]$ or $u=b[n]$ with $b \in W_{\{1,\ldots,n-2\}}$. In the first case
\begin{align*}
x&=[\dots ,n-1,n-2,n,n-1-\mu \nearrow n-1]\\
&=[\dots ,n-\mu-1 \nearrow n-4,n-1,n-2,n,n-3,n-2,n-1],
\end{align*}
and so $[n-1],[n-3,n]^{[n-2,n-1]} \in \Delta_1(w_0x)$. 
In the second, either $b=1$ or there exists $k \in \{1,\ldots,n-2\}$ such that $\ell([k]b)=\ell(b)-1$.
Thus by Lemma~\ref{lem:uv}(ii) $\Delta_1(w_0x)$ contains $[n-1],[n-\mu-1]$ if $b=1$, and $[n-1],[k]$ otherwise.

Finally assume that (iv) holds. If $a+\lambda=n-2$ or $a+\lambda=n-1$ and $u=[\ldots,n-1,n-2,n]$ then it can be checked that $[n],[a] \in \Delta_1(w_0x)$. 
Hence we may assume that $a+\lambda=n-1$ and by applying Lemma~\ref{Dncosetreps} to $u^{-1}$ that $u=b[n]$ for some $b \in W_{\{1,\ldots,n-2\}}$. Therefore there exists $k \in \{1,\ldots,n-2,n\}$ such that $\ell([k]u)=\ell(u)-1$. By Lemma~\ref{lem:uv}(ii) $[k],[a] \in \Delta_1(w_0x)$. Thus if $k \neq a$ then the results follows. Hence assume that $k=a$ and by \cite[Algorithm B]{GeckPfeiffer} $u=[\dots, a \nearrow n-2,n]$. Hence
\begin{align*}
x&=[\dots, a \nearrow n-2,n,a-\mu \nearrow a-1,n-1 \searrow a]\\
&=[\dots, a-\mu \nearrow a-2,a \nearrow n-2, n,a-1,n-1 \searrow a]\\
&=[\dots, a-\mu \nearrow a-2][n,a-1,n-1]^{[n-1 \searrow a]}.
\end{align*}
Thus $[a], [n,a-1,n-1]^{[n-1 \searrow a]} \in \Delta_1(w_0x)$, completing the result.
\end{proof}

\begin{proof}[Proof of Lemma~\ref{lem:pendantDn}] Let $x \in W$ be such that $w_0x \in \I(W)$.
If $x \notin L$, then it follows by Lemmas~\ref{2orseq} and \ref{lem:Dnatleast2} that $|\Delta_1(w_0x)| \geq 2$. 

For $w_0x$ with $x \in \{[1],\ldots,[n-2]\}$, the result follows by Lemma~\ref{lem:w0r}. Let $x=[n,n-2,n-1]$ and let $w \in \Delta_1(w_0x)$. By Lemma~\ref{lem:equiv} $\ell(xw)=\ell(x)-\ell(w)$, and so $w \in \I(\langle [n-2],[n-1],[n] \rangle) = \{[n-2],[n-1],[n],[n-2,n-1,n-2],[n-2,n,n-2]\}$. It is easily checked that $[n-1]$   is the only possibility for $w$.  The result for $[n-1,n-2,n]$ follows similarly, and so the lemma holds.
\end{proof}

\begin{proof}[Proofs of Theorem~\ref{thm:pendantelts} and Corollary~\ref{cor:lwn}]
The finite irreducible Coxeter groups with non-trivial centre are covered by Lemma~\ref{lem:pendantw0central}. The remaining cases are type $\A{n}$, type $\D{n}$ with $n$ odd, type $\mathrm{E}_6$ and type $I_2(m)$ with $m$ odd. These are resolved, respectively, by Lemmas~\ref{lem:pendantAn}, \ref{lem:pendantDn}, \ref{lem:I2m}, and \ref{E6}. 
\end{proof}

\section{Exceptional finite Coxeter groups}\label{sec:exceptional}
Many of the results in this paper had their origins in extensive \textsc{Magma} calculations which revealed unexpected patterns in the structure of their excess zero graphs. In order to facilitate future research on these graphs, we have included below the valency data for four of the small exceptional finite Coxeter groups.

\begin{table}[h]
\begin{tabular}{|c|c|}
\hline
Group & Valency distribution\\
\hline
$W(\H{3})$ & $0^1.1^3.2^4.3^5.4^4.5^5.7^2.8^2.9^2.15^3 $\\
\hline
$W(\H{4})$ & $0^1. 1^{4}.2^{8}.3^{12}.4^{23}.5^{23}.6^{27}.7^{26}.8^{38}.9^{24}.10^{19}.11^{23}.12^{25}.13^{22}.14^ {30}.15^{30}.16^{16}.17^{14}.18^{18}.19^{12}.$\\
& $20^{7}.21^{15}.22^{10}.23^{5}.24^{11}.25^{8}.26^{5}.27^{8}.28^{7}.29^{3}.30^{5}.31^{6}.
32^{4}.33^{5}.34^{3}.35^{2}.36^{3}.37^1.38^2.39^3.$\\
& $40^1.42^3.43^4.44^1.45^2.46^2.47^2.48^2.49^1.50^5.51^1.54^3.55^1. 59^1.61^2.62^3.63^1.65^2.70^1.$\\

& $71^1.79^1.81^1.82^1.83^1.87^2.89^2.97^2.99^1.119^1.122^1.137^2.143^3.173^2.285^4 $\\
\hline
$W(\F{4})$ & $0^1.1^4.2^8.3^9.4^{17}.5^{11}.6^9.7^9.8^{10}.9^{13}.10^2.11^6.12^4.13^4.14^4.15^2.17^4.18^1.19^2.21^2.23^2.25^2.$\\
&$29^2.30^2.34^3.37^2.69^4$\\
\hline
$W(\mathrm{E}_6)$ & $0^1.1^6.2^{14}.3^{34}.4^{35}.5^{33}.6^{15}.7^{56}.8^{30}.9^{58}.10^{25}.11^{52}.12^{17}.13^{28}.14^{19}.15^{25}.16^{25}.17^{24}.18^{11}.19^{40}.$ \\
& $20^{15}.21^{16}.22^5.23^9.24^{12}.25^{13}.26^{13}.27^{18}.28^{14}.29^{7}.30^{11}.31^6.32^4.33^{14}.35^3.36^2.37^{11}.38^{12}.$\\
& $39^6.40^2.41^6.43^{13}.44^9.45^5.46^7.47^7.49^1.50^2.51^2.53^2.54^4.56^2.58^2.59^4. 62^2.68^3.69^1.70^2.$\\
& $71^2.72^4.73^2.74^2.75^1.77^4.80^2.82^{10}.89^9.91^1.95^2.118^6.137^1.141^2.155^7.171^5.227^{10}.445^6 $\\
\hline
\end{tabular} \caption{Valency distribution for small exceptional groups} \label{table:tabEx}
\end{table}

\pagebreak


\begin{thebibliography}{9}

\bibitem{magma} W. Bosma, J. Cannon and C. Playoust, {\em The \textsc{Magma} algebra system. I. The user language}, J. Symbolic Comput., 24 (1997), 235--265.

\bibitem{carter} R.W. Carter, {\em Conjugacy Classes in the Weyl Group}, Compositio. Math. {\bf 25}, Facs. 1 (1972), 1--59.

\bibitem{GeckPfeiffer} M. Geck and G. Pfeiffer, {\em Characters of Finite Coxeter Groups and Iwahori-Hecke Algebras}, LMS Monographs, New Series 21, (2000).

\bibitem{Hansson} M. Hansson and A. Hultman, {\em A word property for twisted involutions in Coxeter groups}, J. Combin. Theory Ser. A {\bf 161} (2019), 220--235; MR3861777

\bibitem{invol} S.B. Hart and P.J. Rowley, {\em Involution products in Coxeter groups.} J. Group Theory 14 (2011), no. 2, 251--259.  


\bibitem{excessminlen} S.B. Hart and P.J. Rowley, {\em Zero excess and minimal length in finite Coxeter groups.} J. Group Theory 15 (2012), no. 4, 497--512.

\bibitem{onexcess} S.B. Hart and P.J. Rowley, {\em On excess in finite Coxeter groups.} J. Pure Appl. Algebra 219 (2015), no. 5, 1657--1669.

\bibitem{excessmaxlen} S.B. Hart and P.J. Rowley, {\em Maximal length elements of excess zero in finite Coxeter groups.} J. Group Theory 21 (2018), no. 5, 817--837.

\bibitem{prefixes} S.B. Hart and P.J. Rowley, {\em A note on involution prefixes in Coxeter groups.} arXiv:2502.00777

\bibitem{Hultman} A. Hultman, {\em The combinatorics of twisted involutions in Coxeter groups.} Trans. Amer. Math. Soc. {\bf 359} (2007), no.~6, 2787--2798; MR2286056

\bibitem{Humph} J.E. Humphreys, \textit{Reflection groups and Coxeter groups.} Cambridge Studies in Advanced Mathematics, 29. Cambridge University Press, Cambridge, 1990.

\bibitem{Marberg} E. Marberg, {\em Extending a word property for twisted Coxeter systems.} Adv. in Appl. Math. {\bf 145} (2023), Paper No. 102477, 26 pp.
\end{thebibliography}
\end{document}